\documentclass[a4paper]{amsart}

\usepackage[english]{babel}
\usepackage[utf8]{inputenc}
\usepackage[T1]{fontenc}
\usepackage{amsmath,amssymb}
\usepackage{amsfonts}
\usepackage{amsthm}
\usepackage{graphicx,epsfig}
\usepackage{comment}

\theoremstyle{plain}
\newtheorem{theorem}{Theorem}
\newtheorem{lemma}[theorem]{Lemma}
\newtheorem{proposition}[theorem]{Proposition}

\theoremstyle{definition}
\newtheorem{definition}{Definition}[section]

\theoremstyle{remark}
\newtheorem*{remark}{Remark}

\numberwithin{equation}{section}

\newcommand\G{\Gamma}
\newcommand{\A}{\mathbf{A}}
\newcommand{\Z}{\mathbf{Z}}
\newcommand{\Q}{\mathbf{Q}}
\newcommand{\C}{\mathbf{C}}
\newcommand{\R}{\mathbf{R}}
\newcommand{\HH}{\mathbf{H}}
\newcommand{\PP}{\mathbf{P}}

\newcommand\g{\gamma}

\newcommand\GG{{\G\backslash G}}

\renewcommand\phi{\varphi}
\newcommand\SL{\mathrm{SL}}
\newcommand\PSL{\mathrm{PSL}}
\newcommand\GL{\mathrm{GL}}
\newcommand\PGL{\mathrm{PGL}}

\title{Equidistribution in $S$-arithmetic and adelic spaces}

\author{Antonin Guilloux}
\address{Sorbonne Universités, UPMC Univ Paris 06, Institut de Mathématiques de Jussieu-Paris Rive Gauche,
UMR 7586, CNRS, Univ Paris Diderot, Sorbonne Paris Cité, F-75005, Paris, France}
\email{antonin.guilloux@imj-prg.fr}
\urladdr{http://webusers.imj-prg.fr/~antonin.guilloux/index.html}

\begin{document}

\begin{abstract}
We give an introduction to adelic mixing and its applications for mathematicians knowing about the mixing of the geodesic flow on hyperbolic surfaces. We focus on the example of the Hecke trees in the modular surface.

Cet article présente une introduction au mélange adélique et ses applications. La présentation faite est pensée pour les mathématiciens connaissant le mélange du flot géodésique sur les surfaces hyperboliques. L'accent est principalement mis sur l'exemple des arbres de Hecke dans la surface modulaire.
\end{abstract}

 \maketitle

 This paper is based on a mini-course given at the conference "Cross-views on hyperbolic geometry and arithmetic" held in Toulouse in November 2012. The purpose was to give an introduction to adelic mixing and its applications for mathematicians who knew about the mixing of the geodesic flow on hyperbolic surfaces; but who may also be intimidated by the $p$-adic and adelic part of the topic. I tried to overcome this  by sticking to the simplest and fundamental example of the Hecke trees in the modular surfaces and by spending some time on taming the concept of adeles. I hope that the description of the solenoid associated to the adeles may help those used to dynamics to get a first insight into the nature of adeles and their dynamical property. I chose to avoid almost entirely the language of algebraic groups, which may be another intimidating topic to go into. Therefore I do not even state the theorem of adelic mixing in its generality. I hope that this introduction will convince geometers that adelic mixing is indeed a natural and interesting tool; and that the description of an example may guide their delving into a more conceptual and comprehensive treatment.
 
 \section*{Introduction}
 
We will introduce some tools to understand the repartition of certain orbits of an action onto a  homogeneous space. The reader should keep in mind throughout the paper --- especially when (s)he does not feel comfortable with the algebraic setting --- that everything begins with the action of a Fuchsian group on the hyperbolic plane. Indeed, we will always deal with:
\begin{itemize}
 \item A locally compact second countable group $G$. The canonical example to keep in mind is $\SL(2,\R)$. Note that $G$ possesses a Haar measure.
 \item A lattice $\Gamma$ of $G$, that is a discrete subgroup such that the Haar measure of $\Gamma\backslash G$ is finite. We will say that $\Gamma$ has finite covolume and denote this volume by
 $\mathrm{covol}(\Gamma)$. The canonical examples are the Fuchsian groups. If asked to pick one of them, say $\SL(2,\Z)$.
 \item A homogeneous space under $G$: the canonical example is the hyperbolic plane $\HH\simeq \SL(2,\R)/\mathrm{SO}(2)$. More generally, one choose a closed subgroup $H$ of $G$ and consider the space $G/H$.
\end{itemize} 
 With these objects, we want to understand the dynamical properties of the action of $\Gamma$ on $G/H$. For example, how many points of an orbit are not to far away from the initial point ? More precisely, in our canonical example, we pick $p\in\HH$ and wonder how many points of the orbit  $\Gamma.p$ lie in a ball $B(p,R)$ of radius $R$ in $\HH$. We will denote by $N_R$ this number:
 $$N_R=\mathrm{Card} \left(\Gamma.p\cup B(p,R)\right).$$
 
 A crucial tool to describe this action is the "duality phenomenon". Indeed, it turns out that the properties of the both the actions 
 $$\Gamma\textrm{ acting on } G/H\textrm{ and }H\textrm{ acting on }\GG$$
 are deeply linked. An easy yet instructive exercise is to note that the orbit of $gH$ under $\Gamma$ is dense in $G/H$ if and only if the one of $\Gamma g$ under $H$ is dense in $\GG$. This observation is the starting point of a huge amount of work. Let us sketch two examples, so that the reader may understand more clearly what lies behind this "duality phenomenon" (and incidentally how the mixing of the geodesic flow appears).
 
 \subsection*{Example 1: Counting and Mixing}
This is the "canonical" example already described:
\begin{itemize}
 \item $G=\SL(2,\R)$.
 \item $\Gamma$ is any lattice of $G$.
 \item The homogeneous space under $G$ is $\HH$. The subgroup $H$ is then $\mathrm{SO(2)}$.
\end{itemize} 
A famous and seminal work of Margulis in his thesis \cite{margulis} was to show a new way of determining the number $N_R$. Indeed he translated, via the "duality phenomenon" the problem in terms of "equidistribution of spheres", i.e. long orbits of $H=\mathrm{SO}(2)$, in the space $\GG$. Recall that the latter is the unitary tangent bundle to the hyperbolic surface $\Gamma\backslash \HH$. These "spheres" are the orbits:
$$\Gamma\backslash \Gamma g H a_t,$$
where the beginning point $p$ of the orbit is the point $gH$ of $\HH=G/H$ and 
$$a_t=\begin{pmatrix} e^{\frac{t}{2}}&0\\0&e^{-\frac{t}{2}}\end{pmatrix}$$
is the geodesic flow. 
These orbits carry a natural probability measure: the image of the Haar measure of $H= \mathrm{SO}(2)$ on $gHa_t$.

This is the point where mixing comes into play: using this property of the geodesic flow (together with an important lemma called "wavefront lemma"), one may prove the equidistribution of spheres. One exactly proves that, for any starting point $p=gH$, the probability measure on  the spheres $\Gamma\backslash \Gamma g H a_t$ converge to the Haar measure on $\GG$. This, in turn, allows to get the estimation:
$$N_R\simeq \frac{\mathrm{Vol}(B(p,R))}{\mathrm{covol}(\Gamma)}.$$
As the mixing of the geodesic flow is exponential one may get an error term. A great quality of this method is that it is very robust: for arithmetic lattices $\Gamma$ (e.g. $\SL(2,\Z)$), number theoretic methods may lead to a very precise estimation of $N_R$. But here, the method works the same for any lattice. Moreover Margulis was able to work in the case of non-constant curvature. And one can adapt it to other $G$ and $H$. The famous paper of Eskin and McMullen \cite{eskin-mcmullen} appears to be a very good entry point in this world.

\subsection*{Example 2: Equidistribution and unipotent flows}
 Here we take:
\begin{itemize}
 \item $G=\SL(2,\R)$.
 \item $\Gamma=\SL(2,\Z)$.
 \item The homogeneous space under $G$ is $\R^2\setminus\{0\}$. The subgroup $H$ is the group of unipotent matrices of the form $\begin{pmatrix} 1&u\\0&1\end{pmatrix}.$
\end{itemize} 
So we are looking at the orbits of $\SL(2,\Z)$ on the euclidean plane, for the linear action. For any $\gamma =\begin{pmatrix}a&b\\c&d\end{pmatrix}\in \Gamma$, we denote by $\|\gamma\|=\sqrt{a^2+b^2+c^2+d^2}$ its euclidean norm and by $\Gamma_T=\{\gamma\in\Gamma \, |\, \|\gamma\|\leq T\} $ the ball in $\Gamma$ of radius $T$. For any point $p$ of $\R^2$, let $\mathrm{Dir}_p$ be the Dirac mass in $p$. Let $\frac{\mathrm{Leb}}{r}$ be the measure on $\R^2\setminus \{0\}$ which in polar coordinates is given by $dr d\theta$. 
Ledrappier \cite{ledrappier} (see also Nogueira \cite{nogueira} for a different approach) showed the following equidistribution result, for any $v=(x,y)\in\R^2$ with $\frac{x}{y}\not\in \Q$:
$$\frac{1}{2T}\sum_{\gamma\in\Gamma_T} \mathrm{Dir}_{\gamma v}\rightharpoonup_{T\to\infty} \frac{\mathrm{Leb}}{r}.$$
For readers not used to these equidistribution statements, this mean that for any continuous function $\phi$ on $\R^2\setminus\{0\}$ with compact support, we have:
$$\frac{1}{2T}\sum_{\gamma\in\Gamma_T} \phi(\gamma v)\xrightarrow{T\to\infty}  \int\int_{\R^2\setminus\{0\}}\frac{\phi(x,y)}{\sqrt{x^2+y^2}}dxdy.$$

In this work too, the key point is the study of the $H$-action on $\GG$. This time, one may use an equidistribution result of Dani-Smillie \cite{dani-smillie}: the dynamic of the unipotent groups in $\GG$ is very rigid and has few invariant ergodic measures. This last result was deeply generalized by Ratner \cite{ratner}, Margulis-Tomanov \cite{margulis-tomanov} (in an $S$-arithmetic setting similar to the one I will introduce afterwards) and eventually by Benoist-Quint \cite{benoist-quint}. Those generalizations in turn lead to generalization of Ledrappier's result by varying $\Gamma$, $G$ and $H$. We let the reader look at Gorodnik-Weiss \cite{gorodnik-weiss} for the real case in a general setting, Maucourant \cite{maucourant}, Maucourant-Weiss \cite{maucourant-weiss} and Guilloux \cite{guilloux1} for some more precisions on Ledrappier's result, and Guilloux \cite{guilloux2} for an $S$-arithmetic treatment. 

\bigskip

As we saw in both example, those techniques are very robust and may be adapted to various groups and even $p$-adic or adelic groups. Those are very interesting for arithmetic or number-theoretic study. I would like to explain here how it is sometimes possible to translate or reinterpret problems of arithmetic flavor in such a way that they look very similar to dynamical or equidistribution problems. I shall concentrate on a problem related to the first example and use mixing.

As a very simple to state example, let us mention that a direct adaptation of the strategy of Margulis for example $1$ together with a result on adelic mixing gives an answer to the following question \cite{guilloux3}, which seems highly arithmetic ($d\geq 3$ and $n$ is an integer):\\
Decide the existence and estimate the number of elements of $\mathrm{SO(3,\Q)}$ of "denominator" $n$, i.e. which may be written as $\frac{1}{n}$ times a matrix with integer entries:
$$\mathrm{SO(3,\Q)}\cap \frac{\mathcal M(3,\Z)}{n}\textrm{ as }n\to\infty.$$

\bigskip

Here is the structure of the paper: in the following section, I will introduce some beautiful objects of arithmetic origin which naturally interplay with hyperbolic geometry. They are well-known at least in some areas of mathematics: Hecke trees. At the end of this section, I will be able to state the theorem we will be interested to: equidistribution of Hecke spheres.

In the second section, I will introduce the algebraic tools needed to reinterpret the problem of equidistribution of Hecke spheres as a dynamical problem related to some "geodesic" flow. I will introduce $p$-adic fields. I include a description of a nice dynamical system, the "solenoid", which helps to build an intuition of these fields. Then I will move on to the adeles, which are nearly a product of all the $p$-adic fields. For an algebraic group defined with equations with rational coefficients, we consider the group of its point over the adeles. We explain briefly the link this last group and the group of real points, thanks to Borel- Harish-Chandra theorem. Then we focus on the group $\mathrm{PGL}(2)$ and describe briefly the tree attached to it and its links with the Hecke tree defined in the previous section.

Eventually, the last section consists in the statement of adelic mixing for $\mathrm{PGL}(2)$ and a sketch of proof of the equidistribution of Hecke spheres.

\section{Hecke trees and Hecke correspondance}

The Hecke trees are the main object of this paper. We will see it appear in several ways. The first one is the more concrete one and we will gradually move to the adelic version of it (in subsection \ref{ss:adelictrees}).

\subsection{Construction of the Hecke trees in the modular surface}\label{ss:constructiontree}

We want to consider $X(1)=\PSL(2,\Z)\backslash \HH$ as the space of similarity classes of lattices in $\C\simeq \R^2$. Let us quickly review how it is done. First of all the identification $\C\simeq \R^2$ is realized through the choice of the canonical basis $(1,i)$ of $\C$. A lattice in $\C$ is a discrete subgroup of rank $2$, i.e. of the form $\Z e_1+\Z e_2$, where $(e_1,e_2)$ is a $\R$-basis of $\C$. A \emph{marked} lattice in $\C$ is a lattice with a chosen basis $(e_1,e_2)$. The space of marked lattices may be identified to $\GL(2,\R)$\footnote{The space of marked lattices is more precisely a principal homogeneous space under $\GL(2,\R)$: it is homogeneous and the stabilizer of any point is trivial. So we have to choose a base point to fix the identification between the space of marked lattices and $\GL(2,\R)$. Here we take the lattice $\Z+\Z i$ associated to the canonical basis $(1,i)$.}, as a marked lattice is uniquely defined by the basis $(e_1,e_2)$. In order to keep coherency in the notations, the action of $\GL(2,\R)$ on lattices is a right one, induced by the action on $\C$:
$$\begin{pmatrix} x&y\\z&t\end{pmatrix}\cdot (a+ib) = ax+by+i(az+bt).$$ 
Now the space of marked lattices \emph{up to homothety} is identified to $\PGL(2,\R)$; and the space of marked lattices \emph{up to similarity} is then naturally identified with 
$$\PGL(2,\R)/\mathrm{O}(2)\simeq \HH.$$
And if you want to forget the marking, you still have to mod out by the stabilizer of the lattice $\Z+\Z i$, that is $\PGL(2,\Z)$. Up to an easy reduction from $\PGL$ to $\PSL$, we are done. Indeed, the space of lattices up to similarity is:
$$\PGL(2,\Z)\backslash \PGL(2,\R)/\mathrm{O}(2)\simeq \PSL(2,\Z)\backslash \HH=X(1).$$

Now choose a prime number $p$ and pick a point $[\Lambda]\in X(1)$ (of course $[\Lambda]$ denotes the class of the lattice $\Lambda$). Consider the set:
$$\{\textrm{lattice }\Lambda'\, | \, \Lambda'<\Lambda\textrm{ has index }p\}.$$ We may describe this set:
\begin{lemma}
An element $\Lambda'$ of the above set is given by the choice of the line:
$$\Lambda'/p\Lambda \textrm{ in } \Lambda/p\Lambda \simeq (\Z/p\Z)^2.$$ 
\end{lemma}
\begin{proof}
Indeed, $p\Lambda$ is included in $\Lambda'$ (because $\Lambda'$ has index $p$ in $\Lambda$), so $\Lambda'$ projects to $\Lambda/p\Lambda$. The projection is a subgroup and we compute its cardinal: the cardinal of $\Lambda'/p\Lambda$ is $$\frac 1p\times\mathrm{Card}\left(\Lambda'/p\Lambda\right)=p.$$
The projection is indeed a line in $\Lambda/p\Lambda$.

Conversely, given a line  $L$ in $\Lambda/p\Lambda$, there is a unique subgroup $\Lambda'$ of $\Lambda$ of index $p$ which projects to this line: as we have seen, we must have $p\Lambda\subset\Lambda'$. So $\Lambda'$ is exactly the preimage of the line $L$ in $\Lambda$ under the projection $\Lambda\to \Lambda/p\Lambda$.
\end{proof}

Hence the number of lattices $\Lambda'$ in the above defined set is $p+1$, as the cardinal of the projective line $\PP((\Z/p\Z)^2).$ Their classes $[\Lambda']$ are called the ($p$-Hecke)-neighbors of $[\Lambda]$ in $X(1)$. A straightforward verification shows that the set of neighbors of some $[\Lambda]$ in $X(1)$ does not depend on the choice of the representative $\Lambda$. Seeing $X(1)$ as $\PSL(2,\Z)\backslash \HH$, we may write this relation explicitly: the neighbors of the class of $z\in\HH$ are the classes of $pz$ and the $\frac{k+z}{p}$, $0\leq k\leq p-1$, see figure \ref{fig:heckeneighbors}.

\begin{figure}
  \includegraphics{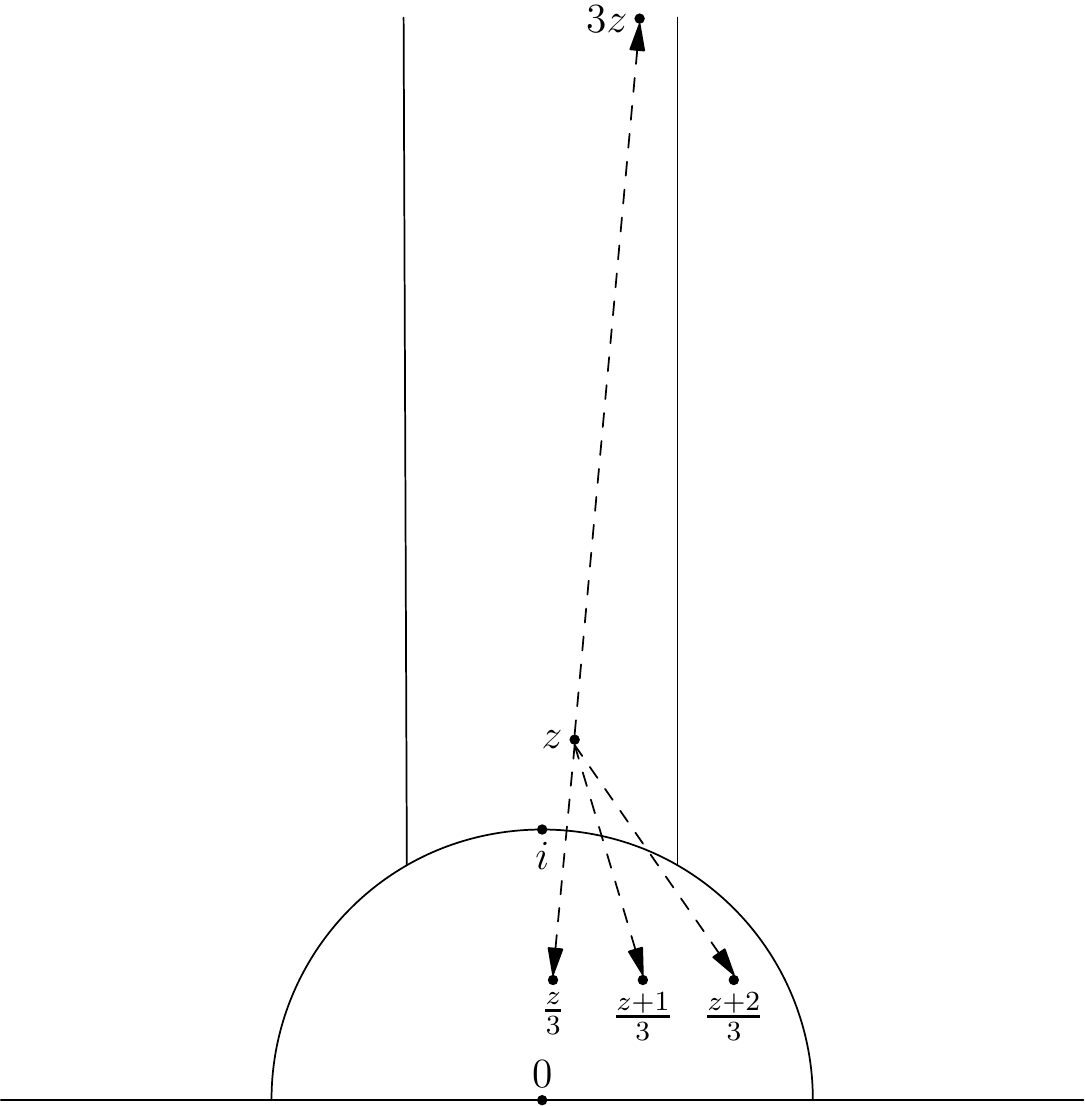}
  \caption{Hecke-neighbors for $p=3$}\label{fig:heckeneighbors}
\end{figure}

\begin{remark}
The neighbors are not automatically distinct: if $\Lambda =\Z+\Z i$, the two lattices
$$\Lambda_1'=\Z p+\Z i \textrm{ and } \Lambda'_2=\Z+\Z pi$$
are both of index $p$ in $\Lambda$ but are in the same similarity class. They project to the same point in $X(1)$. We will sometimes add multiplicities to deal with that.
\end{remark}

Moreover the "neighbor" relation is a reflexive one. Indeed, if the class $[\Lambda']$ is a neighbor of $[\Lambda]$, then we may choose the representatives so that $\Lambda'$ is a sublattice of index $p$ in $\Lambda$. But, then, $p\Lambda$ is a sublattice of index $p$ in $\Lambda'$. And the class $[p\Lambda]=[\Lambda]$ is a neighbor of $[\Lambda']$. At this point, the neighbor relation constructs a graph.

Eventually, one may prove that this relation has no cycle. The idea is that, if 
$$\Lambda_0 < \Lambda_1<\ldots<\Lambda_k$$
is a sequence of lattices such that each time $\Lambda_i<\Lambda_{i+1}$ has index $p$, then either there exists $i$ such that $p^2 \Lambda_i=\Lambda_{i+2}$ (we are backtracking in the graph) or
there are two vectors $e_1$ and $e_2$ in $\Lambda_0$ such that:
$$\textrm{For all }i,\; \Lambda_i=\Z e_1+\Z p^{-i}e_2.$$
In the second case, $\Lambda_0$ and $\Lambda_k$ are not similar.
This is not obvious but rather elementary (see \cite{serre2}).

\medskip

Consider $\tau_p$ the set of vertices of the complete tree of valence $p+1$, rooted at some vertex $t_p^0$. The above construction defines a map $h_p^{[\Lambda]}$ from $\tau_p$ to $X(1)$, sending $t_p^0$ to $[\Lambda]$ and neighbors in the tree to neighbors\footnote{This map is really only defined up to isometries of the tree fixing the root $t_p^0$. This subtlety will not interfere and we will not be more precise at this point. This map will be formally defined in section \ref{ss:adelictrees}. Beware that it may be non-injective as in the example of $\mathbb Z[i]$.} in $X(1)$ without backtracking. Note eventually that the relation "to be a sublattice of index $p$" commutes to the action of $\textrm{O}(2)$. Hence we might as well work in the unit tangent bundle $\PSL(2,\Z)\backslash \PSL(2,\R)$ of $X(1)$. And we get a mapping of the tree $\tau_p$ in $\PSL(2,\Z)\backslash \PSL(2,\R)$.

Let us now briefly describe another point of view, which makes clearer how one may vary the group $G$ and the lattice $\Gamma$. The matrix $g_p=\begin{pmatrix} p&0\\0&1\end{pmatrix}$ belongs to the commensurator of $\PSL(2,\Z)$: the group $g_p\PSL(2,\Z)g_p^{-1}\cap \PSL(2,\Z)$  is a subgroup of finite index in both $\PSL(2,\Z)$ and $g_p\PSL(2,\Z)g_p^{-1}$. More precisely, the elements of $g_p\PSL(2,\Z)$ fall into $p+1$ different classes modulo $\PSL(2,\Z)$. One easily check that:
$$\PSL(2,\Z)g_p\PSL(2,\Z)=\PSL(2,\Z)\{g_p;\; \begin{pmatrix}1&k\\0&p\end{pmatrix}\textrm{ for }0\leq k\leq p-1\}.$$
The $p+1$ lattices $g_p.(\Z+i\Z)$ and $\begin{pmatrix}1&k\\0&p\end{pmatrix}.(\Z+i\Z)$ are exactly the neighbors of $\Z+i\Z$: check they are all distinct and of index $p$ in $\Z+i\Z$.
In other terms, the neighbors of $[\Z+\Z i]$ are exactly the elements of 
$$\PSL(2,\Z)\backslash \PSL(2,\Z)g_p\PSL(2,\Z).$$ 
For another lattice $\Gamma$ of $\PSL(2,\R)$, or of another group $G$, a notion of neighbors may still be defined for any element $g$ of the commensurator of $\Gamma$: the commensurator of $\Gamma$ is the group of elements $g$ such that $g\Gamma g^{-1}\cap \Gamma$ has finite index in $G$ and $\g\Gamma g^{-1}$. It is exactly the property needed for the above construction. And the whole construction described above makes sense. It is especially interesting for groups $\Gamma$ with big commensurator, that is arithmetic lattices in $G$ (such as $\PGL(2,\Z)$).

\subsection{Product of trees and spheres}

Fix a point $[\Lambda]\in X(1)$. For each prime number $p$, we have constructed a mapping:
$$h_p^{[\Lambda]} \: :\: \tau_p \to X(1).$$
Recall that the tree $\tau_p$ is rooted at $t_p^0$. And we denote by $d$ the distance in these trees given by the number of edges between two points.
We would like to consider as a whole the set of prime numbers. Indeed, keep in mind that we are interested in arithmetics. 
And we will use and abuse of the prime factorization of integers. 
Let us make a seemingly silly remark: in the factorization of each integer into primes, only a finite set of prime numbers appears;
Of course not a bounded one, but still finite. 
So, we will not exactly consider the set of prime numbers but work with its finite subsets. 

We consider the \emph{restricted} product of the trees $\tau_p$:
$$\tau := \left\{(t_p)_p\in \prod_{p\textrm{ prime}} \tau_p \: | \: \sum_p d(t_p,t_p^0) <\infty\right\}.$$
A point in $\tau$ has all its entries equal to $t_p^0$ but a finite set. Then we construct a mapping
$$h^{[\Lambda]} \: : \: \tau \to X(1)$$ in the following way. Consider a point$(t_p)_p$ in $\tau$. Let $x_0=[\Lambda]$, $x_2=h^{x_0}_2(t_2)$. 
And, recursively, for a prime number $p$, with $q$ the greatest prime number less than $p$, $x_p=h^{x_q}_p(t_p)$. 
This sequence indexed by the prime numbers becomes eventually stationary: as $(t_p)$ belongs to the restricted product, for $q$ large enough, $t_q$ always equals $t_q^0$ and the point $x_q$ does not move anymore: with the above notation, the next point in this sequence is $x_p=h^{x_q}_p(t^0_p)=x_q$.
We define $h^{[\Lambda]}$ to be this limit $x_p$. 
One may prove that the order in which we construct the sequence does not change the limit point.

It is maybe clearer to look at the spheres in the product of trees $\tau$. Let $N$ be an integer and $N=\prod_p p^{\nu_p}$ its factorization into prime factors. Remark that $\sum_p \nu_p<\infty$. So the set 
$$S_N:=\left\{(t_p)_p\in \prod_p \tau_p \: | \: \forall p,\: d( t_p,t_p^0)=\nu_p\right\}$$ is a subset of $\tau$. We call it the "sphere of radius $N$". 
Remark that we may give a more concrete interpretation: $h^{[\Lambda]}(S_N)$ is the set of classes $[\Lambda']$, where $\Lambda'$ is a sublattice of index\footnote{And I mean \emph{really} of index $N$ : that is $\Lambda'$ is a sublattice of index $N$ in $\Lambda$ and no lattice $\frac 1k\Lambda'$, $k\in\mathbb N$, is a sublattice of $\Lambda$.} $N$ in $\Lambda$: indeed, pick a point $(t_p)$ in $S_N$. Let us follow the sequence defining $h^{[\Lambda]}((t_p))$:
\begin{itemize}
\item $x_0=[\Lambda]$
\item if $q,p$ are two successive prime numbers, $x_p=h^{x_q}(t_p)$ is a sublattice of index\footnote{Same remark as above.} $p^{\nu_p}$ in $x_q$: recall that $t_p$ is at distance $\nu_p$ from the root, so in order to construct $x_p=h^{x_q}(t_p)$ you take $\nu_p$ times a neighbor, i.e. $\nu_p$ times a sublattice of index $p$ without backtracking.
\end{itemize}
At the end, $h^{[\Lambda]}((t_p))$ is a sublattice of index\footnote{Same remark as above.} $\prod_p p^\nu_p=N$ in $[\Lambda]$. And by letting $(t_p)$ vary in the sphere of radius $N$, you get every such sublattice.

The Hecke correspondence $T_N$ on $X(1)$ is the operation which associates to $[\Lambda]$ the set $h^{[\Lambda]}(S_N)$\footnote{This last set is weighted whenever $h^{[\Lambda]}$ is not injective. Later on, any sum on the elements of $T_N([\Lambda])$ is to be understood as weighted.}.

\subsection{Dynamical problems}

\subsubsection{Distribution of spheres}\label{ss:Heckecorrespondance}

The question is: if $N\to \infty$, how do the sets $T_N([\Lambda])$ look like ?

This question has been answered in many ways and is known to be directly related to "Ramanujan conjecture" \cite{sarnak}. In the case explained above, it should be attributed to Linnik-Skubenko \cite{linnik,skubenko}.
\begin{theorem}\label{th:equi1}
These sets equidistribute toward the hyperbolic area $\mu$ on $X(1)$:
$$\frac{1}{\textrm{Card}(S_N)}\sum_{t\in S_N} \mathrm{Dirac}_{h^{[\Lambda]}(t)} \rightharpoonup \mu.$$

In other terms, for any continuous and compactly supported function $\phi$ on $X(1)$, we have:
$$\frac{1}{\textrm{Card}(S_N)}\sum_{t\in S_N} \phi(h^{[\Lambda]}(t)) \to \int_{X(1)} \phi d\mu.$$
\end{theorem}

This theorem has many generalizations (varying $G$, $\Gamma$ and $H$) and also many proofs. The main proofs follow:
\begin{itemize}
\item harmonic analysis (cf. Sarnak \cite{sarnak}),
\item use of adelic mixing (cf. Clozel-Oh-Ullmo \cite{clozel-oh-ullmo}...)
\item use of Ratner theorem for unipotent flows (Eskin-Oh \cite{eskin-oh}, Duke-Rudnick-Sarnak \cite{duke-rudnick-sarnak}, Eskin-Mozes-Shah \cite{eskin-mozes-shah}...)
\end{itemize}
We will try to explain here the second approach.

\subsubsection{Distribution of closed geodesics}

Less fundamental in a number-theoretic point of view but still natural for geometers is the second question. This one takes place at the level of the unit tangent bundle $\PSL(2,\Z)\backslash \PSL(2,\R)$ of $X(1)$.

Suppose $x \in \PSL(2,\Z)\backslash \PSL(2,\R)$ induces a closed geodesic. We will see later on that any $y$ in the image $h^x(\tau)$ of the restricted product of trees $\tau$ still induces a closed geodesic. So one may wonder:\\
What can said be about these geodesics when $y$ drift apart $x$ in $h^x(\tau)$.

We will not go into this question, but with the present introduction the reader may study the paper of Aka-Shapira \cite{aka-shapira}.

\bigskip

At this point, it is not clear how these objects and questions are related to the examples of the introduction: we see that $\PSL(2,\Z)$ and $\PSL(2,\R)$ keep appearing but who is $H$ ? Which dynamical system is mixing ?
Understanding this requires the introduction of $p$-adic fields and adeles. We do this in the next section, trying to keep the most dynamical point of view on these notions. Once we are acquainted with those algebraic objects, both questions may naturally be reinterpreted as evolution of sets under some kind of geodesic flow. We will state a mixing property for it; it will lead us to the answer to the first question.

\section{$p$-adic and adelic groups}

\subsection{$p$-adic fields}

I will not develop here the theory of $p$-adic fields (see \cite{serre} for an introduction).  We begin, of course, by fixing a prime number $p$.

Let us define the $p$-adic field $\Q_p$:
\begin{definition}
$\Q_p$ is the completion of $\Q$ for the $p$-adic absolute value $|\cdot|_p$.
\end{definition}

In this form, it may be abstract but it shows a fundamental feature of $p$-adic fields: they are counterparts of the real numbers $\R$. As $\R$, they are completion of $\Q$ for some absolute value. And Ostrowski's theorem states that the usual one and the $|\cdot|_p$ for $p$ prime are the only absolute values on $\R$. So, in order to reflect in a proper way properties of $\Q$ -- that is arithmetic -- in complete fields, you have to consider at once $\R$ and all the $\Q_p$. That is why the adeles were introduced. But before going on, let us present $p$-adic fields in a more concrete way, in order to get a bit of intuition.

First of all, the $p$-adic absolute value on $\Q$ is so defined: for a rational number $r\in \Q$, write it in the form $r=p^k \frac a b $, with $a$ and $b$ coprime with $p$. Then:
$$|r|_p=|p^k \frac a b|_p=p^{-k}.$$
Note that if $r$ is an integer, then $|r|_p\leq 1$: $\Z$ becomes bounded with this absolute value. Its completion $\Z_p$ will then be compact. Let us define it before $\Q_p$: let 
$$\Z_p=\{(x_n)_{n\geq 0} \: | \: x_n\in \Z/p^n\Z\textrm{ and } x_{n+1}=x_n \, (\textrm{mod }p^n)\}.$$ 
We may define the absolute value $|\cdot |_p$ in $\Z_p$: 
$$|(x_n)_n|_p=p^{-k},$$ where $k${ is the greatest integer such that $x_j=0$ for all $j\leq k$.
For this absolute value, it is not hard to check that $\Z_p$ is a complete and compact set.
We have a natural injection $\left\{ \begin{matrix} \Z&\to &\Z_p \\ k&\mapsto & (k \, (\textrm{mod }p^n))\end{matrix}\right.$ and one checks that this is an isometry and the image is dense. This prove that $\Z_p$ is indeed the completion of $\Z$.
Topologically, $\Z_p$ is a Cantor set.

As follows from the definition, $\Z_p$ is a ring. Its fraction will be a complete field in which $\Q$ embeds isometrically for the absolute value $|\cdot|_p$ with  dense image. Hence this fraction field is the completion $\Q_p$. One can also write:
$$\Q_p=\bigcup_{n\geq 0} p^{-n} \Z_p.$$
With this last presentation, one sees that $\Z_p\cap\Z[\frac1p]=\Z$.
Let us insist on the fact that, from the $p$-adic point of view, $p^{-n}$ is bigger and bigger as $n$ grows.

\begin{remark}
For simplicity of notation, we will sometimes denotes $\R$ by $\Q_\infty$ and its usual absolute value by $|\cdot|_\infty$.
\end{remark}

We will denote by $\mathcal V$ the set of "places" of $\Q$, i.e. of different completions. We will write $$\mathcal V=\{\infty\}\cup\{\textrm{prime numbers }p\}.$$ Given the convention in the above remark and Ostrowski's theorem, the $\Q_\nu$'s for $\nu\in\mathcal V$ are the only completions of $\Q$. There is a very elegant and simple formula that relates these different completions, called the product formula:
$$\textrm{For any }x\in\Q\setminus\{0\}\textrm{ we have }\prod_{\nu\in\mathcal V} |x|_\nu =1.$$

\subsection{$S$-arithmetic rings and solenoids}

Let us take a finite subset $S$ containing $\infty$. We define:
$$\Q_S := \prod_{\nu\in S}\Q_\nu\textrm{ and }\Z_S=\Z\left[\{\frac1p ,\, p\in S\setminus\{\infty\}\}\right].$$
We have the diagonal injection of $\Q$ in these $\Q_S$. As $\Z_S\subset \Q$, it naturally embeds in $\Q_S$ (and even any $\Q_{S'}$ for any other $S'$ regardless to the relation between $S$ and $S'$).
Those objects, called $S$-arithmetic, allow to keep track of properties of rational numbers regarding the powers of the prime numbers $p\in S$ appearing. For example $\Q_S$, $S=\{\infty,2,3\}$ is the right place to study the rational solutions to an equation with denominator highly divisible by $2$ or $3$. The following theorem holds:
\begin{theorem}
\begin{itemize}
\item $\Q_S$ is a locally compact ring,
\item $\Z_S\subset\Q_S$ is a discrete cocompact subgroup. A fundamental domain is $[0;1[\times \prod_{p\in S\setminus\{\infty\}}\Z_p$.
\item If $S$ is a proper subset of a finite subset $S'\in \mathcal V$, $\Z_{S'}$ is dense in $\Z_S$. 
\end{itemize}
\end{theorem}

Let us take some time to describe a dynamical system attached to this object, which will be a solenoid in the dynamical sense: a fibration over a circle by a Cantor set, whose monodromy has dense orbits in any fiber. For this description we take $S=\{\infty,p\}$, but it works the same with different choices. Recall that with the above definition, $\Z\left[\frac1p\right]$ is embedded in $\R\times \Q_p$ diagonally. Consider the space 
$$Y=\Z\left[\frac1p\right] \backslash \; (\R\times \Q_p).$$ 
We want to define  a projection $\Pi: Y \to \Z\backslash \R$ from $Y$ to the circle. The following lemma is the key point:
\begin{lemma}
Consider a point $(x_\infty,x_p)\in\R \times \Q_p$. Then there exists $z\in \Z\left[\frac1p\right]$ such that
$$\left\{z'\in\Z\left[\frac1p\right] \: | \: z'+x_p\in \Z_p\right\}=z+\Z.$$ 
\end{lemma}
\begin{proof}
The last point in previous theorem implies that $\Z\left[\frac1p\right]$ is dense in $\Q_p$. So there is some $z$ such that $z+x_p$ lies in the open subset $\Z_p$. Moreover, if some $z'$ also verifies $z'+x_p\in \Z_p$, then we have ($\Z_p$ is a ring):
$$z-z'=(z+x_p)-(z'+x_p)\in \Z\left[\frac1p\right]\cap \Z_p=\Z.$$
\end{proof}
We may define the image under $\Pi$ of $\Z\left[\frac1p\right]+(x_\infty,x_p)$: it is the class $\Z+(x_\infty+z)$ for any $z$ given by the lemma. We may check that $\Pi$ is a fibration over the cercle, with fibers isomorphic to $\Z_p$.

There is a natural flow on $Y$: lift to $Y$ the flow $x\to x+t$ on the circle. Let us describe it precisely.
Choose a point $p=\Z\left[\frac1p\right]+(x_\infty,x_p)$ in $Y$, and for sake of simplicity, suppose we have chosen (with the lemma above) $x_p\in \Z_p$. So $\Pi(p)=\Z + x_\infty$. Then for all $t\in \R$ define 
$$p_t:=\Z\left[\frac1p\right] + (x_\infty+t,x_p).$$
Remark the projection $\Pi(p_t)$ is $\Z + (x_\infty+t)$. Moreover, using the action of $\Z$, we may write ($\{t\}$ and $\lfloor t\rfloor$ are the fractional and integral parts of $t$):
$$p_t=\Z\left[\frac1p\right] + (x_\infty+\{t\},x_p-\lfloor t\rfloor ).$$
\begin{figure}
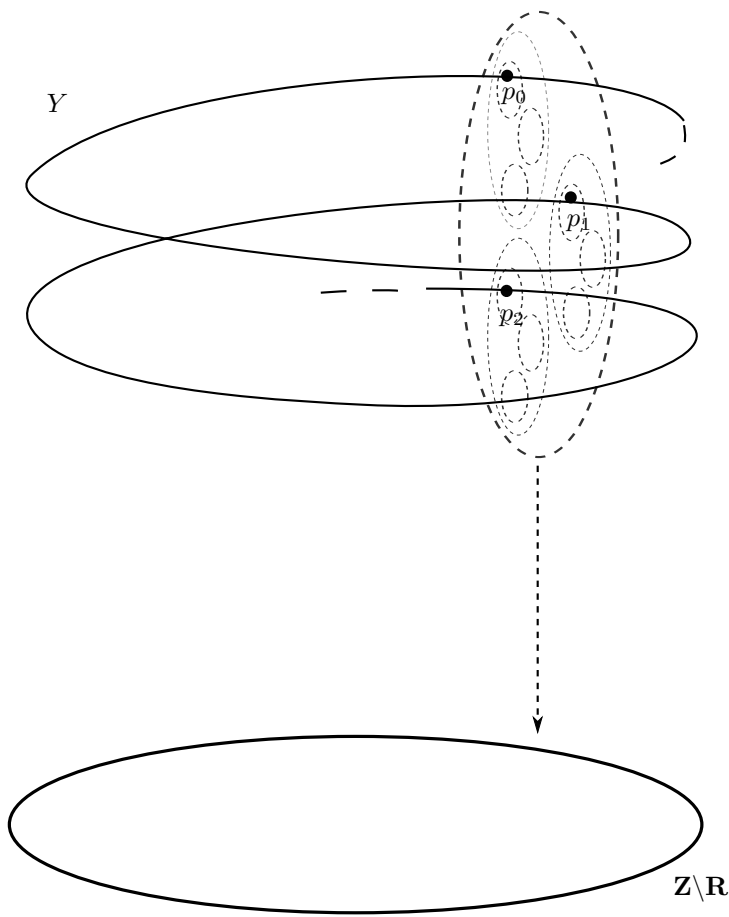
\caption{A solenoid}\label{fig:sol}
\end{figure}
So we return in the fiber of $p$ at each integer time, but we move in this fiber. This motion is the monodromy of the fibration and is given by:
$$p_n=\Z\left[\frac1p\right] + (x_\infty,x_p-n).$$
As $\Z$ is dense in $\Z_p$, the orbit $p_n$ is dense in the fiber of $p$.

All this may be summarized in a double quotient:
$$\Z\left[\frac1p\right] \backslash \R\times \Q_p / \Z_p =\Z\backslash \R.$$
By varying the compact $\Z_p$, one may vary the circle. For example, if we had taken $p\Z_p$ instead, we would have got $p\Z \backslash \R$.

\subsection{Adeles}

Until now, we were only able to consider a finite number of prime numbers. This is not natural or very interesting from an arithmetical point of view. One would like to consider all primes at once. But considering the mere product of all the $\Q_\nu$'s is not a good solution: as $\Z_S$ was a discrete cocompact subgroup of $\Q_S$, we would like $\Q$ to be a discrete cocompact subgroup of the ring constructed. However the product would give a non locally compact ring.

So, as for the trees, we will do a restricted product: considering the diagonal embedding of $\Q$ in $\prod_\nu \Q_\nu$, we may remark that the image of any rational number is always in $\Z_p$ for $p$ large enough (any $p$ not appearing in the prime factorization of the denominator of the rational number). So we will restrict the product to such elements: their components should eventually be in $\Z_p$. This is how the adeles $\A$ are defined:
$$\A:=\left\{(x_\nu)\in \prod_{\nu\in\mathcal V} \Q_\nu \: | \: \textrm{all but finitely many }x_\nu\textrm{ belong to }\Z_\nu\right\}.$$
The topology on $\A$ is generated by open sets of the form:
$$\prod_{\nu\in S}O_\nu\times \prod_{\nu\not\in S}\Z_\nu,$$
where $S\subset \mathcal V$ is finite, and each $O_\nu$ is an open subset of $\Q_\nu$.

With this definition, we may extend the theorem given in the $S$-arithmetic case:
\begin{theorem}
\begin{itemize}
\item $\A$ is a locally compact ring,
\item $\Q \subset \A$ is a discrete cocompact subgroup. A fundamental domain is $$[0;1[\times \prod_{p\in \mathcal V\setminus\{\infty\}}\Z_p.$$
\item For every prime number $p$, $\Q$ is dense in $\A_{(p)}=\Q_p\backslash \A$. 
\end{itemize}
\end{theorem}

A crucial consequence of the first point is that $\A$ has a Haar measure. It is defined (up to normalizations) as the product of Lebesgue measure and the Haar measures of the $\Q_p$.

It will be helpful to understand what does "going to $\infty$" mean for a sequence in $\A$. Formally, it means that we are leaving any compact set. But let us have a look at two sequences which go to $\infty$ in a very different way:
\begin{itemize}
\item Going to $\infty$ "vertically": for example the sequence $u_k=2^{-k}$. Here we see that a projection on  $\Q_2$ is going to $\infty$.
\item Going to $\infty$ "horizontally": for example the sequence $(v_k)_{k\geq 0}$ where $v_k$ is the $k$-th prime number. Then, for all prime $p$, as soon as $v_k>p$, the projection of $v_k$ to $\Q_p$ is inside $\Z_p$. But this sequence is going to $\infty$ because the greatest prime number $q$ for which it does not project inside $\Z_q$ goes to $\infty$ with $k$.
\end{itemize}
Now the simple sequence $v_k=\frac1k$ presents a mixed behavior: it sure goes to $\infty$, but sometimes 
(at big powers of a fixed prime number $p$) in the vertical direction: $\frac1{p^n}$ is big in $\Q_p$, but in $\Z_q$ 
for $q\neq p$; and sometimes (when $k=q$ is a prime number) in the horizontal direction : $\frac{1}{q}$ is not so big 
in $\Q_q$, but as $q=k$ goes to $\infty$, it becomes big in $\A$. From now on, we will use the notation $x\to \infty$ to denote the property "$x$ leaves any compact set".

It is also worth noting that these construction extends without difficulty to the case of number fields instead of $\Q$. 

\subsection{Groups}

Having constructed several ring, we may consider the groups of points of an algebraic groups over them. I do not want to go into any formal definition of algebraic groups here. See \cite{platonov-rapinchuk} for a reference on algebraic groups and their $S$-arithmetic or adelic points. For us it will be a subgroup of $\mathrm{GL}(n)$ defined by actual equations with coefficients in $\Z$; and also $\PGL(2)$, even if it does not belong to the previous class. Other examples are the classical groups $\SL(2)$, $\mathrm{SO}(2)$...

Given such a group $G$, we will denote by $G_\Z$, $G_\Q$ ... (more generally $G_R$ for a ring $R$) the set of elements of $\mathrm{GL}(n,\Z)$, $\mathrm{GL}(n,\Q)$, $\mathrm{GL}(n,R)$ verifying the equations. All our rings are topological and locally compact; that imply in turn that $G_R$ is topological and locally compact. We will be interested in the relationship between $G_\Z$ and $G_\R$, $G_{\Z_S}$ and $G_{\Q_S}$ ($S$ finite) and $G_\Q$ and $G_\A$. We have the following theorem, due to Borel and Harish-Chandra. Recall first that a \emph{lattice} in a locally compact is a discrete subgroup such that the quotient has finite volume for the Haar measure. It is cocompact if the quotient is compact. 
\begin{theorem}[Borel and Harish-Chandra]
Fix any finite subset $S\subset \mathcal V$ containing $\infty$. We have equivalence between the statements:
\begin{enumerate}
\item $G_\Z$ is a lattice $G_\R$ (resp. cocompact),
\item $G_{\Z_S}$ is a lattice $G_{\Q_S}$ (resp. cocompact),
\item $G_\Q$ is a lattice $G_\A$ (resp. cocompact).
\end{enumerate}
\end{theorem}
Let us mention that the theorem gives a criterion for both case, and that if $G$ is a semisimple group, then $G_\Q$ is a lattice in $G_\A$. Our main interest will be in the group $\PGL(2)$ which is indeed semisimple. 
So the three statements are true (but the lattices are not cocompact, as seen for $\PGL(2,\Z)\subset \PGL(2,\R)$).

Moreover, when $G_\Z$ is a lattice in $G_\R$, one may understand the double quotients:
$$G_\Z\backslash G_\R \simeq G_{\Z\left[\frac1p\right]}\backslash G_\R\times G_{\Q_p} / G_{\Z_p}\simeq G_\Q\backslash G_\A / \prod_p G_{\Z_p}.$$
In this form, we may describe $G_\Q\backslash G_\A$ as a solenoid over the base $G_\Z\backslash G_\R$ (i.e. a fibration over $G_\Z\backslash G_\R$ with fiber Cantor sets and a dense monodromy of $G_\Z$).

\subsection{Adelic interpretation of Hecke trees}\label{ss:adelictrees}

Here we link the previous section with the technology we have just briefly presented: we interpret moving in the Hecke trees as a generalization of the geodesic flow, i.e. as the action of a diagonal matrix of $\PGL(2,\A)$. The first step is to see that, in some sense, the tree $\tau_p$ is an analogue to the hyperbolic disc: it identifies with the quotient of $\PGL(2,\Q_p)$ by the maximal compact subgroup $\PGL(2,\Z_p)$.
\begin{proposition}
The quotient $\PGL(2,\Q_p)/\PGL(2,\Z_p)$ has a natural structure of a $(p+1)$-complete tree.
\end{proposition}
We refer to Serre \cite{serre2} for a comprehensive presentation of these trees. We will just give a very sketchy idea of the proof. The main idea is that it is very much similar to the construction of the Hecke tree.

First of all, $\PGL(2,\Z_p)$ is open and closed in $\PGL(2,\Q_p)$. So the quotient is a discrete set. Moreover, $\PGL(2,\Z_p)$ is the stabilizer of (the homothety class of) the "maximal order" $\Z_p^2$ of $\Q_p^2$ (a maximal order is a free $\Z_p$-module of rank $2$). If $[\Lambda]$ and $[\Lambda']$ are two such orders, we say that they are "neighbors" if (up to the choice of suitable representatives), we have:
$$p\Lambda \subset \Lambda'\subset \Lambda,$$ with each inclusion being of index $p$.

As for the Hecke tree, once fixed $\Lambda$, such a $\Lambda'$ is determined by the choice of a line in 
$$p\Lambda \backslash \Lambda \simeq (p\Z \backslash \Z)^2.$$
So $[\Lambda]$ has $p+1$ neighbors. The relation is symmetric and has no cycle, giving to the quotient the structure of a tree.

One may say that $\PGL(2,\Q_p)$ is the unitary tangent bundle to this tree in the sense that each element of $\PGL(2,\Q_p)$ corresponds bijectively to an oriented marked geodesic: let $g=(v_1 \; v_2)$ be an element of $\mathrm{GL}(2,\Q_p)$ ($v_1$ and $v_2$ are two non-colinear elements of $\Q_p^2$). Then to $g$ is naturally associated the order $\Z_p v_1+\Z_p v_2$, that is a point of the tree. Moreover, one attach to $g$ the following geodesic:
$$\left\{[\Z_p p^n v_1+ \Z_p v_2], n\in \Z\right\}.$$
This geodesic is marked at $[\Z_p v_1+\Z_p v_2]$ and oriented "toward $v_2$": as $n\to\infty$ the first component becomes more and more negligible and asymptotically vanishes.

The geodesic flow is given by going one-step forward in this geodesic. From the previous description of the geodesic, it appears that this flow is given by the right action of the diagonal matrix
$$h_p:=\begin{pmatrix} p&0\\0&1\end{pmatrix}.$$ And the sphere of radius $n$ around some point $g\PGL(2,\Z_p)$ is the subset\footnote{Compare with the alternative point of view on the Hecke correspondence at the end of subsection \ref{ss:constructiontree}.}:
$$g\PGL(2,\Z_p)h_p^n\PGL(2,\Z_p)\textrm{ of }\PGL(2,\Q_p)/\PGL(2,\Z_p).$$

The similarity between the two constructions of the tree is not fortuitous. Indeed, the Hecke tree really is the projection of the tree of $\PGL(2,\Q_p)$ through the double quotient of previous section. Let us describe this projection. From now on we fix an identification of $\tau_p$ with $\PGL(2,\Q_p)/\PGL(2,\Z_p)$ sending the root of $\tau_p$ to the class of $\mathrm{Id}$.

Choose a point $x=\PGL(2,\Z)g$ in the unit tangent bundle $\PGL(2,\Z)\backslash \PGL(2,\R)$ to $X(1)$.
Then we may identify $\tau_p$ with $\{g\}\times \PGL(2,\Q_p)/\PGL(2,\Z_p)$; and thus consider it as a subset of $\PGL(2,\R)\times \PGL(2,\Q_p)/\PGL(2,\Z_p)$.
The map $h^x_p \: : \: \tau_p \to \PGL(2,\Z)\backslash \PGL(2,\R)$ is given by the following diagram:
$$\begin{matrix} \tau_p &\hookrightarrow & \PGL(2,\R)\times \PGL(2,\Q_p)/\PGL(2,\Z_p)\\
&&\downarrow \\
& & \PGL(2,\Z\left[\frac 1p\right])\backslash \PGL(2,\R)\times \PGL(2,\Q_p)/\PGL(2,\Z_p)
\end{matrix}$$
The last space is, as stated in the previous subsection, identified with the unit tangent bundle $\PGL(2,\Z)\backslash \PGL(2,\R)$.

\begin{remark}
The image of $\tau_p$ under $h^x_p$ does not depend on the choice of the representative $g$. But two different choices leads to two different maps $h^x_p$ differing by an automorphism of $\tau_p$ at the source.
\end{remark}

Until now, we have looked at only one prime number $p$. One can perform the same analysis as above in the adelic group $\PGL(2,\A)$, in order to get a description of the image of the spheres in $X(1)$. Let $x=\PGL(2,\Z)g\mathrm{SO}(2)$ be a point in $X(1)$, let $N$ be an integer and define the element $h_N\in \PGL(2,\A)$ whose component in $\PGL(2,\R)$ is $\mathrm{Id}$ and $\begin{pmatrix} N&0\\0&1\end{pmatrix}$ in each $\PGL(2,\Q_p)$. Then we have
\begin{lemma}\label{lem:spheres}
The sphere $h^x(S_N)$ of radius an integer $N$ with center a point $x=\PGL(2,\Z)g\mathrm{SO(2)}$ in $X(1)$ is the image of the set:
$$\left( g\mathrm{SO}(2)\times \prod_p \PGL(2,\Z_p)\right) h_N \prod_p \PGL(2,\Z_p)$$
in $X(1)$ via the identification
$$X(1)=\PGL(2,\Z)\backslash \PGL(2,\R)/\mathrm{SO}(2)\hspace{3cm}$$
$$\hspace{3cm}=\PGL(2,\Q)\backslash \PGL(2,\A)/\left(\mathrm{SO}(2)\times \prod_p \PGL(2,\Z_p)\right).$$
\end{lemma}

\section{Adelic mixing and Equidistribution of Hecke spheres}

The previous section explained a technology to reinterpret the dynamical problems related to Hecke trees. In this section, we explain the strategy to deduce the answer to the first question: equidistribution of Hecke spheres, see subsection \ref{ss:Heckecorrespondance}. This strategy builds over a very deep result, namely \emph{adelic mixing}, that I will present. 
I will try to convince the reader that, once accepted this result, the equidistribution of Hecke spheres is rather "easy". 
Of course, almost everything is hidden in the adelic mixing property. But I hope that the reader more or less used to the usual mixing of the geodesic flow will admit this mixing property without too much difficulties.
And this presentation might convince this reader that ideas coming out from dynamical considerations, once properly reinterpreted, leads to nice and natural arithmetic considerations.

\subsection{Adelic mixing}

Let us recall the setting of subsection \ref{ss:Heckecorrespondance}. We take a point $x=[\Lambda]$ in $X(1)$ and write it as $\PGL(2,\Z)g\mathrm{SO}(2)$. Stated in another way, $\PGL(2,\Z)g$ is a point in the unitary tangent bundle projecting to $x$.
We can lift $x$ one more step upward, that is inside $\PGL(2,\Q)\backslash \PGL(2,\A)$: $x$ is then the projection of the set $$\left( g\mathrm{SO}(2)\times \prod_p \PGL(2,\Z_p)\right)$$ in the double quotient $$X(1)=\PGL(2,\Q)\backslash \PGL(2,\A) /\left(\mathrm{SO}(2)\times \prod_p \PGL(2,\Z_p)\right).$$

Moreover, thanks to lemma \ref{lem:spheres}, the Hecke spheres $h^x(S_N)$ are the projection in this double quotient of 
$$\left( g\mathrm{SO}(2)\times \prod_p \PGL(2,\Z_p)\right)\cdot h_N.$$
In order to prove equidistribution, we will prove that, already at the level of $\PGL(2,\Q)\backslash\PGL(2,\A)$, this sets equidistribute toward the Haar probability measure on $\PGL(2,\Q)\backslash\PGL(2,\A)$. Recall that, up to normalization, a Haar measure on $\PGL(2,\A)$ exists and is unique. Moreover, as $\PGL(2,\Q)$ is a lattice in $\PGL(2,\A)$, it induces a unique probability measure on $\PGL(2,\Q)\backslash\PGL(2,\A)$; we will call it the volume and denote it by $\mathrm{Vol}$. Now what exactly does this equidistribution mean ?

Consider the Haar probability measure $\nu$ on the compact subgroup $$\mathrm{SO}(2)\times \prod_p \PGL(2,\Z_p)$$ of $\PGL(2,\A)$. It projects to a probability measure $\bar\nu$ in $\PGL(2,\Q)\backslash\PGL(2,\A)$. So we are looking at the evolution of $\bar\nu$ under the transformation $\PGL(2,\Q)g\mapsto \PGL(2,\Q)gh_N$ in $\PGL(2,\Q)\backslash \PGL(2,\A)$ and would like to prove it converges to $\textrm{Vol}$.

At this point, it appears clearly that a mixing property for the action of $h_N$ on $\PGL(2,\Q)\backslash \PGL(2,\A)$ with respect to $\mathrm{Vol}$ would help: this property implies that, for any measure $\mu$ absolutely continuous with respect to $\mathrm{Vol}$, the measures $(h_N)^*\mu$ converge to $\mathrm{Vol}$. 
The good news is that this mixing property holds.

Let us define some notations: the space $\PGL(2,\Q)\backslash \PGL(2,\A)$ has a probability measure $\mathrm{Vol}$, so it makes sense to look at the space $\mathrm{L}^2\left(\PGL(2,\Q)\backslash \PGL(2,\A)\right)$ of square integrable functions on it. We denote by $\langle\cdot,\cdot\rangle$ the hermitian scalar product. Moreover, we denote by $\mathrm{L}_0^2\left(\PGL(2,\Q)\backslash \PGL(2,\A)\right)$ the subspace of functions $\phi$ such that $\int \phi d\mathrm{Vol}=0$. The action of the group $\PGL(2,\A)$ on the quotient $\PGL(2,\Q)\backslash\PGL(2,\A)$ leaves the measure $\mathrm{Vol}$ invariant: the latter is the projection of the Haar measure. So this action yields an action on the space of square-integrable functions. This action is defined, with obvious notations, by: 
$$g\cdot \phi \: :\: x\mapsto \phi(x.g).$$
We may then state:
\begin{theorem}\label{thm:mixing}
Let  $\phi$ and $\psi$ be two functions of $\mathrm {L}_0^2\left(\PGL(2,\Q)\backslash \PGL(2,\A)\right)$. We suppose that $\phi$ and $\psi$ are invariant under the action of $\prod_p \PGL(2,\Z_p)$.

Then, as $g\to\infty$ in $\PGL(2,\A)$, we have:
$$\langle \phi ,g\cdot \psi\rangle \to 0.$$
\end{theorem}

This theorem deserves a lot of comments. 

First of all, it is indeed a mixing property: if $A$ and $B$ are two open sets of $\PGL(2,\Q)\backslash\PGL(2,\A)$, denoting by $\phi=\chi_A-\mathrm{Vol}(A)$ and $\psi=\chi_B-\mathrm{Vol}(B)$ their normalized characteristic functions , we get the usual mixing property:
$$\mathrm{Vol}(A\cap B g^{-1})\xrightarrow{g\to\infty}\mathrm{Vol}(A)\mathrm{Vol}(B).$$

Second, it is valid in a very wide generality. It is more or less true for any semisimple algebraic group. We refer the reader to Gorodnik-Maucourant-Oh \cite{gorodnik-maucourant-oh} for a general result.

It is proven by the study of certain unitary representations of $\PGL(2,\A)$. In this setting, it states that the trivial representation is isolated among automorphic ones. This is called "property $\tau$" (in reference to the stronger property $T$). The final step of this was done by Clozel \cite{clozel}, after many works. I will not go into these topics, far too involved for my purpose.

One last comment: the speed of convergence (so-called "decay of coefficients") is known. In the case of $\PGL(2)$ it is deeply and directly related to bounds towards the "Ramanujan conjecture".

The reader will find more explanations in Gorodnik-Maucourant-Oh \cite{gorodnik-maucourant-oh}, Clozel-Oh-Ullmo \cite{clozel-oh-ullmo}, Sarnak \cite{sarnak} and Goldstein-Mayer \cite{goldstein-mayer}.

\subsection{How to prove equidistribution of Hecke spheres ?}

As previously analyzed, we want to prove the convergence of the measure $h_N^*\nu$. First of all, note that the sequence $h_N$ of elements of $\PGL(2,\A)$ goes to $\infty$ as $N\to \infty$. So they fit in the setting of theorem \ref{thm:mixing}. However, the measure $\bar\nu$ is too singular: it is supported on a set of volume $0$. It is a classical trick to first approximate $\bar\nu$ by a measure which has a density with respect to $\mathrm{Vol}$. And then hope the approximation will not be too difficult to track when letting the dynamics evolve.

Here we are in a very good situation. Recall that the support of $\bar\nu$ is the set:
$$\PGL(2,\Q)g\left(\mathrm{SO}(2)\times \prod_p \PGL(2,\Z_p)\right).$$
Now, the real component of $h_N$ is $\mathrm{Id}$. So one can "fatten up" $\mathrm{SO}(2)$, by taking a small neighborhood $\Omega$ of $\mathrm{Id}$ in $\PGL(2,\R)$; and replace $\nu$ by the volume restricted to $g\left(\mathrm{SO}(2)\Omega\times \prod_p \PGL(2,\Z_p)\right)$, normalized to be a probability measure. Denote by $\nu_\Omega$ this probability measure. As $\Omega$ shrinks to $\mathrm{SO}(2)$, $\nu_\Omega$ converges to $\nu$. But, note that the action of $\Omega$ commutes to the action of $h_N$:
$$\PGL(2,\Q)g\left(\mathrm{SO}(2)\Omega\times \prod_p \PGL(2,\Z_p)\right)h_N\hspace{3cm}$$
$$\hspace{3cm}=\PGL(2,\Q)g\left(\mathrm{SO}(2)\times \prod_p \PGL(2,\Z_p)\right)h_N\Omega.$$
So the fattening is inert under the dynamic: $h_N^*\bar\nu_\Omega$ converges uniformly to $h_N^*\bar\nu$.

So we only need to prove the convergence of $h_N^*\bar \nu_\Omega$ towards $\mathrm{Vol}$, for any $\Omega$.
 This is a direct consequence of theorem \ref{thm:mixing}, as $\bar \nu_\Omega$ has a (bounded) density with respect to $\mathrm{Vol}$:
  let $\psi$ be this density. 
  For any function $\phi$ continuous with compact support on $\PGL(2,\Q)\backslash\PGL(2,\A)$, let $\phi'=\phi-\int\phi d\mathrm{Vol}$ and $\psi'=\psi-1$. Then we have:
\begin{eqnarray*}
  \int \phi d(h_N^*\bar\nu_\Omega)&=&\int \phi' d(h_N^*\bar\nu_\Omega) +\int \phi d\mathrm{Vol}\\
  &=&\int \phi' (h_N \cdot\psi)d\mathrm{Vol}+\int \phi d\mathrm{Vol}\\
  &=&\langle \phi',h_N\cdot \psi\rangle+\int \phi d\mathrm{Vol}\\
  &=&\langle \phi',h_N\cdot \psi'\rangle+\int \phi d\mathrm{Vol}\\
  &\xrightarrow{N\to\infty}& \int \phi d\mathrm{Vol}
\end{eqnarray*}  
From the third line to the fourth, you just use that $\langle \phi',1\rangle=\int \phi' d\mathrm{Vol}=0$. And the convergence is given by the mixing property, i.e. theorem \ref{thm:mixing}.

This proves theorem \ref{th:equi1} on the equidistribution of Hecke spheres.


\bigskip

This text only touch on the topic of adelic mixing and adelic dynamical systems. For instance, I completely ignored any property linked to entropy of these dynamic; the reader may refer to the beautiful papers by Einsiedler-Lindenstrauss-Michel-Venkatesh \cite{ELMV} (where Hecke trees and their adelic interpretation are a central object), Lindenstrauss \cite{lindenstrauss}, Einsiedler-Katok-Lindenstrauss \cite{EKL}.

\bibliographystyle{alpha}
\bibliography{biblio}

\def\cprime{$'$}
\begin{thebibliography}{ELMV11}

\bibitem[AS12]{aka-shapira}
Menny Aka and Uri Shapira.
\newblock On the evolution of continued fractions in a fixed quadratic field.
\newblock {\em arXiv preprint arXiv:1201.1280}, 2012.

\bibitem[BQ11]{benoist-quint}
Yves Benoist and Jean-Fran{\c{c}}ois Quint.
\newblock Mesures stationnaires et ferm\'es invariants des espaces homog\`enes.
\newblock {\em Ann. of Math. (2)}, 174(2):1111--1162, 2011.

\bibitem[Clo03]{clozel}
Laurent Clozel.
\newblock D\'emonstration de la conjecture {$\tau$}.
\newblock {\em Invent. Math.}, 151(2):297--328, 2003.

\bibitem[COU01]{clozel-oh-ullmo}
Laurent Clozel, Hee Oh, and Emmanuel Ullmo.
\newblock Hecke operators and equidistribution of {H}ecke points.
\newblock {\em Invent. Math.}, 144(2):327--351, 2001.

\bibitem[DRS93]{duke-rudnick-sarnak}
W.~Duke, Z.~Rudnick, and P.~Sarnak.
\newblock Density of integer points on affine homogeneous varieties.
\newblock {\em Duke Math. J.}, 71(1):143--179, 1993.

\bibitem[DS84]{dani-smillie}
S.~G. Dani and John Smillie.
\newblock Uniform distribution of horocycle orbits for {F}uchsian groups.
\newblock {\em Duke Math. J.}, 51(1):185--194, 1984.

\bibitem[EKL06]{EKL}
Manfred Einsiedler, Anatole Katok, and Elon Lindenstrauss.
\newblock Invariant measures and the set of exceptions to {L}ittlewood's
  conjecture.
\newblock {\em Ann. of Math. (2)}, 164(2):513--560, 2006.

\bibitem[ELMV11]{ELMV}
Manfred Einsiedler, Elon Lindenstrauss, Philippe Michel, and Akshay Venkatesh.
\newblock Distribution of periodic torus orbits and {D}uke's theorem for cubic
  fields.
\newblock {\em Ann. of Math. (2)}, 173(2):815--885, 2011.

\bibitem[EM93]{eskin-mcmullen}
Alex Eskin and Curt McMullen.
\newblock Mixing, counting, and equidistribution in {L}ie groups.
\newblock {\em Duke Math. J.}, 71(1):181--209, 1993.

\bibitem[EMS96]{eskin-mozes-shah}
Alex Eskin, Shahar Mozes, and Nimish Shah.
\newblock Unipotent flows and counting lattice points on homogeneous varieties.
\newblock {\em Ann. of Math. (2)}, 143(2):253--299, 1996.

\bibitem[EO06]{eskin-oh}
Alex Eskin and Hee Oh.
\newblock Ergodic theoretic proof of equidistribution of {H}ecke points.
\newblock {\em Ergodic Theory Dynam. Systems}, 26(1):163--167, 2006.

\bibitem[GM03]{goldstein-mayer}
Daniel Goldstein and Andrew Mayer.
\newblock On the equidistribution of {H}ecke points.
\newblock {\em Forum Math.}, 15(2):165--189, 2003.

\bibitem[GMO08]{gorodnik-maucourant-oh}
Alex Gorodnik, Fran{\c{c}}ois Maucourant, and Hee Oh.
\newblock Manin's and {P}eyre's conjectures on rational points and adelic
  mixing.
\newblock {\em Ann. Sci. \'Ec. Norm. Sup\'er. (4)}, 41(3):383--435, 2008.

\bibitem[Gui08]{guilloux3}
Antonin Guilloux.
\newblock Existence et \'equidistribution des matrices de d\'enominateur {$n$}
  dans les groupes unitaires et orthogonaux.
\newblock {\em Ann. Inst. Fourier (Grenoble)}, 58(4):1185--1212, 2008.

\bibitem[Gui10a]{guilloux1}
Antonin Guilloux.
\newblock A brief remark on orbits of {${\rm SL}(2,\mathbb Z)$} in the
  {E}uclidean plane.
\newblock {\em Ergodic Theory Dynam. Systems}, 30(4):1101--1109, 2010.

\bibitem[Gui10b]{guilloux2}
Antonin Guilloux.
\newblock Polynomial dynamic and lattice orbits in {$S$}-arithmetic homogeneous
  spaces.
\newblock {\em Confluentes Math.}, 2(1):1--35, 2010.

\bibitem[GW07]{gorodnik-weiss}
Alex Gorodnik and Barak Weiss.
\newblock Distribution of lattice orbits on homogeneous varieties.
\newblock {\em Geom. Funct. Anal.}, 17(1):58--115, 2007.

\bibitem[Led99]{ledrappier}
Fran{\c{c}}ois Ledrappier.
\newblock Distribution des orbites des r\'eseaux sur le plan r\'eel.
\newblock {\em C. R. Acad. Sci. Paris S\'er. I Math.}, 329(1):61--64, 1999.

\bibitem[Lin68]{linnik}
Yu.~V. Linnik.
\newblock {\em Ergodic properties of algebraic fields}.
\newblock Translated from the Russian by M. S. Keane. Ergebnisse der Mathematik
  und ihrer Grenzgebiete, Band 45. Springer-Verlag New York Inc., New York,
  1968.

\bibitem[Lin06]{lindenstrauss}
Elon Lindenstrauss.
\newblock Invariant measures and arithmetic quantum unique ergodicity.
\newblock {\em Ann. of Math. (2)}, 163(1):165--219, 2006.

\bibitem[Mar69]{margulis}
G.~A. Margulis.
\newblock Certain applications of ergodic theory to the investigation of
  manifolds of negative curvature.
\newblock {\em Func. Anal. Appl.}, 4:333--335, 1969.

\bibitem[Mau07]{maucourant}
Fran{\c{c}}ois Maucourant.
\newblock Homogeneous asymptotic limits of {H}aar measures of semisimple linear
  groups and their lattices.
\newblock {\em Duke Math. J.}, 136(2):357--399, 2007.

\bibitem[MT94]{margulis-tomanov}
G.~A. Margulis and G.~M. Tomanov.
\newblock Invariant measures for actions of unipotent groups over local fields
  on homogeneous spaces.
\newblock {\em Invent. Math.}, 116(1-3):347--392, 1994.

\bibitem[MW12]{maucourant-weiss}
Fran{\c{c}}ois Maucourant and Barak Weiss.
\newblock Lattice actions on the plane revisited.
\newblock {\em Geom. Dedicata}, 157:1--21, 2012.

\bibitem[Nog02]{nogueira}
Arnaldo Nogueira.
\newblock Orbit distribution on {$\mathbb R^2$} under the natural action of
  {${\rm SL}(2,\mathbb Z)$}.
\newblock {\em Indag. Math. (N.S.)}, 13(1):103--124, 2002.

\bibitem[PR94]{platonov-rapinchuk}
Vladimir Platonov and Andrei Rapinchuk.
\newblock {\em Algebraic groups and number theory}, volume 139 of {\em Pure and
  Applied Mathematics}.
\newblock Academic Press Inc., Boston, MA, 1994.
\newblock Translated from the 1991 Russian original by Rachel Rowen.

\bibitem[Rat94]{ratner}
M.~Ratner.
\newblock Invariant measures and orbit closures for unipotent actions on
  homogeneous spaces.
\newblock {\em Geom. Funct. Anal.}, 4(2):236--257, 1994.

\bibitem[Sar91]{sarnak}
Peter~C. Sarnak.
\newblock Diophantine problems and linear groups.
\newblock In {\em Proceedings of the {I}nternational {C}ongress of
  {M}athematicians, {V}ol.\ {I}, {II} ({K}yoto, 1990)}, pages 459--471, Tokyo,
  1991. Math. Soc. Japan.

\bibitem[Ser73]{serre}
J.-P. Serre.
\newblock {\em A course in arithmetic}.
\newblock Springer-Verlag, New York, 1973.
\newblock Translated from the French, Graduate Texts in Mathematics, No. 7.

\bibitem[Ser03]{serre2}
Jean-Pierre Serre.
\newblock {\em Trees}.
\newblock Springer Monographs in Mathematics. Springer-Verlag, Berlin, 2003.
\newblock Translated from the French original by John Stillwell, Corrected 2nd
  printing of the 1980 English translation.

\bibitem[Sku62]{skubenko}
B.~F. Skubenko.
\newblock The asymptotic distribution of integers on a hyperboloid of one sheet
  and ergodic theorems.
\newblock {\em Izv. Akad. Nauk SSSR Ser. Mat.}, 26:721--752, 1962.

\end{thebibliography}
\end{document}